\documentclass[11pt]{article}
\date{}
\usepackage{amssymb}
\usepackage{amsmath}
\usepackage{amsthm}
\usepackage{graphicx}
\usepackage{indentfirst}
\allowdisplaybreaks[4]
\newtheorem{theorem}{Theorem}

\newtheorem{corollary}[theorem]{Corollary}

\newtheorem{definition}[theorem]{Definition}

\newtheorem{lemma}[theorem]{Lemma}

\newtheorem{remark}[theorem]{Remark}

\numberwithin{equation}{section}
\numberwithin{theorem}{section}
\newcommand{\keywords}[1]{\par\noindent\textbf{Keywords:} #1}
\newcommand{\subjclass}[2]{
  \par\noindent\textbf{Mathematics Subject Classification} #2}

\begin{document}

\title{Nonexistence of solutions to $\Delta_pu+\Delta_qu+u^s|\nabla u|^t\leq 0$ on geodesically complete noncompact Riemannian manifolds}

\author{Biqiang Zhao\\
\small Beijing International Center for Mathematical Research,
\\
\small Peking University, Beijing, China\\ \small
biqiangzhao123@outlook.com}
            
\maketitle
\setlength{\parindent}{2em}

\begin{abstract}
In this paper, we consider the inequality $\Delta_pu+\Delta_qu+u^s|\nabla u|^t\leq 0$ on geodesically complete noncompact Riemannian manifolds. By a test function argument, we establish Liouville-type theorems under the upper bound of volume of geodesic ball. In the Euclidean space $\mathbb{R}^n$, we obtain new nonexistence results which extend the result of Bhakta-Biswas-Filippucci \cite{BBF}. In particular, we have addressed the influence of the higher order term for $s<0$. At last, we present some examples to illustrate sharpness in some cases.

\end{abstract}

\keywords{ Liouville-type theorem,  quasilinear inequality,  $(p,q)$-Laplacian}
\subjclass [{ 35J60, 35J92, 53C20, 53A55 }


\section{Introduction}
\label{1}
 In recent years, the study of nonexistence results for weak solutions to differential equations and inequalities has attracted much attention \cite{AS,BBF2,BMP,BMGV,Fi1,GK,MMP1,MMP2,MMP3}. The purpose of the present paper is to study the Liouville type theorems for weak solution of some differential inequalities with $(p,q)$-Laplace operator involving gradient nonlinearities on geodesically complete noncompact Riemannian manifolds.
 \par
 This problem has a long history. In 1844, Cauchy \cite{Ca} published the first statement of what is now known as the Liouville theorem for bounded analytic functions. Since then many improvements or generalizations have been generalized to more equations or inequalities. In the seminal paper \cite{GS}, Gidas and Spruck proved that there are no positive solution to the equation
 \begin{align}
     \Delta u+u^p=0,\quad  \mathrm{in}\ \mathbb{R}^n,\quad n>2 
 \end{align}
 if 
 \begin{align*}
 1<p<\frac{n+2}{n-2}.
 \end{align*}
 Moreover, there exists no nontrivial nonnegative supersolution of (1.1) provided that
\begin{align*}
    1<p\leq \frac{n}{n-2}.
 \end{align*}
 \par
 Gidas and Spruck's result can be generalized to $p$-Laplacian. Serrin and Zou \cite{SZ} proved that there are no positive solution to the classical Lane-Emden-Fowler equation
 \begin{align}
     \Delta_pu+u^\alpha=0, \quad \mathrm{in}\  \mathbb{R}^n
 \end{align}
 if
 \begin{align*}
     1<p<n, \quad 1<\alpha<\frac{(n+1)p-n}{n-p}.
 \end{align*}
Mitidieri and Pohozaev \cite{MP1} proved that there exists no nontrivial 
nonnegative supersolution of (1.2) provided that 
\begin{align*}
    1<\alpha\leq \frac{(p-1)n}{n-p},\quad 1<p<n.
\end{align*}
Moving to exterior domains, Bidaut-V$\acute{\mathrm{e}}$ron and Pohozaev \cite{BMP} showed that the nonneagtive supersolution of (1.2) is $u\equiv 0$ if $ p<n$ and $ 1<\alpha\leq \frac{(p-1)n}{n-p}$ or $ p=n$ and $1<\alpha<\infty$. In \cite{MP2}, Mitidieri and Pohozaev first proved the Liouville property for the inequality involving a gradient nonlinearity of the form
\begin{align}
    \Delta_p u+u^s|\nabla u|^m\leq 0,\quad  \mathrm{in}\ \mathbb{R}^n
\end{align}
when the exponents belong to the subcritical range given by
    \begin{align*}
    s(n-p)+m(n-1)<n(p-1), \quad s+m>p-1.
\end{align*}
The cases $s(n-p)+m(n-1)=n(p-1) $ and $ 0<s\leq p-m-1$ were proved by Filippucci \cite{Fi1,Fi2}.
\par
 For general Riemannian manifolds, Sun, Xiao and Xu \cite{SXX} established a sharp Liouville principle for the weak solutions to the quasilinear elliptic inequality  (1.3) for $(p,s,m)\in (1,\infty)\times \mathbb{R}\times \mathbb{R}$. In particular, they obtained nonexistence results for $(s,m)\in  (-\infty,0)\times (-\infty,0)$. The approach in \cite{SXX} has a more recent history. Inspired by \cite{Ku}, Grigor’yan and Kondratiev \cite{GK} and Grigor’yan and Sun \cite{AS} studied the differential inequality of the
form
\begin{align}
    \Delta u+u^\sigma\leq 0
\end{align}
on a Riemannian manifold $M$. Particularly in \cite{AS}, Grigor’yan and Sun proved that (1.4) has no nontrivial nonnegative solution if for some $x_0\in M$ and large enough $r$,
\begin{align*}
    \mathrm{vol} B(x_0,r)\leq C r^{\frac{2\sigma}{\sigma-1}}\mathrm{ln}^{\frac{1}{\sigma-1}}r,
\end{align*}
where $B(x_0,r) $ is the geodesic ball. They also showed the exponents $\frac{2\sigma}{\sigma-1}$ and $\frac{1}{\sigma-1} $ are sharp. Later, Mastrolia, Monticelli and Punzo \cite{MMP1} investigated a class of differential inequalities with a potential and showed that the potential function gives a direct influence on the nonexistence of nonnegative solutions.
\par 
In this paper, we are concerned with the differential inequality with $(p,q)-$Laplacian operator, i.e.,
\begin{align}
    \Delta_pu+\Delta_qu+u^s|\nabla u|^t\leq 0
\end{align}
  on a geodesically complete noncompact Riemannian manifold $M$. The idea of studying such operators comes from the problems of the calculus of variations and nonlinear elasticity theory, cf. \cite{Ma1,Ma2,Zh1,Zh2}. For example, the $(p,q)-$Laplace operator is related to the study of reaction-diffusion systems
  \begin{align*}
      u_t=\mathrm{div}(A(u)\nabla u)+c(x,u).
  \end{align*}
 The $(p,q)$-Laplace operator can be obtained by the form $A(u)=|\nabla u|^{p-2}+|\nabla u|^{q-2}$. In the last few years, the analysis aspects of $(p,q)$-Laplace operator have many achievements. In \cite{BT}, Bobkov and Tanaka studied the existence and nonexistence of positive solutions for the $(p,q)$-Laplace equations with two parameters in a bounded domain. Wang and Zhang \cite{WZ} obtained the gradient estimates for solutions to nonlinear elliptic equation driven by the $(p,q)$-Laplace operator. Recently, Bhakta, Biswas and Filippucci \cite{BBF} established several Liouville-type theorems for differential quasilinear inequalities with $(p,q)$-Laplace in the entire $\mathbb{R}^n$ (or an exterior domain). In \cite{ZZ}, Zhou and Zhu improved the Serrin’s index range in \cite{BBF} by the vector field method. In \cite{BBF2}, Bhakta, Biswas and Filippucci obtained Liouville theorems for ($p,q)$-Laplace elliptic equations with source terms involving gradient nonlinearity. In \cite{Zhao}, the author studied the nonnegative solutions of the differential inequality with $(p,q)$-Laplacian operator on Riemannian manifolds and generalized the result of Mastrolia-Monticelli-Punzo \cite{MMP1} to $(p,q)$-Laplace operator. In \cite{zhao2}, the author established Liouville-type theorems for parabolic
differential inequalities with $(p,q)$-Laplacian operator on Riemannian
manifolds under the weighted volume growth assumptions.
\par
Throughout the paper, we assume that $1<q\leq  p$ and $M$ is a geodesically complete noncompact Riemannian manifold equipped with the Riamannian distance $d(\cdot,\cdot)$ and measure $\mu$. Denote by $V(r)$ the volume of the geodesic ball $B(x_0,r)$ centered at $x_0$ with radius $r$. Since the constant $C>0$ is not important, it may vary at different occurrences. First, we give the definition of the weak solution.
\begin{definition}
    Let $p>q>1,\ (s,t)\in \mathbb{R}\times \mathbb{R}$. We say that $u\in C^{1}(M)$ is a positive weak solution of (1.5) if $u>0$, $u^s|\nabla u|^t\in L_{loc}^{1}(M)$ and for every $0\leq \psi\in W^{1,p}(M)\cap L^{\infty}(M)$ with compact support, one has
    \begin{align}
        -\int_M|\nabla u|^{p-2}\langle \nabla u,\nabla \psi\rangle d\mu -\int_M|\nabla u|^{q-2}\langle \nabla u,\nabla \psi\rangle d\mu+\int_Mu^s|\nabla u|^t\psi d\mu \leq 0.
    \end{align}
\end{definition}
    Next we define the following regions $\{G_i\}_{i=1}^4$ (see Section 2) by
    \begin{align*}  
    \begin{cases}
        G_1=\{(s,t)|\  s\geq 0,t>p-s-1 \}; \\
           \\
           G_2=\{(s,t) |\ s<0,\frac{s}{q-1}+\frac{t}{p-1}>1\}; \\
           \\
           G_3=\{(s,t)|\ t>q-1,\frac{s}{p-1}+\frac{t}{q-1}<1\} ;\\
           \\
           G_4=\{(s,t)|\ t\leq q-1,s<q-t-1 \}.
    \end{cases}   
\end{align*}
    Now we state the main results.
    \begin{theorem}   
     For $(s,t)\in G_1$, (1.5) admits no nontrivial positive weak solution if one of the following assumptions holds:
 
 ~\\
   (a)  $(s,t)\in G_{1,a}=G_1\cap\{(s,t)|\ t<q\}$ and there is 
   \begin{align}
       V(r)\leq C r^{\frac{qs+t}{s+t-q+1}}(\mathrm{ln}\ r)^{\frac{q-1}{s+t-q+1}}, \quad \forall r>>1;
   \end{align}
   \\
   (b)  $(s,t)\in G_{1,b}=G_1\cap\{(s,t)|\ t\geq q\}$ and there is 
   \begin{align}
       V(r)\leq C r^{q}(\mathrm{ln}\ r)^{q-1}, \quad \forall r>>1.
   \end{align}    
    \end{theorem}

 \begin{theorem}
     For $(s,t)\in G_2$, (1.5) admits no nontrivial positive weak solution if one of the following assumptions holds:
     
     ~\\
     (a)  $(s,t)\in G_{2,a}=G_2\cap\{(s,t)|\ s\geq -1,t\geq q\}$ and there is 
  \begin{align}
       V(r)\leq C r^{q}(\mathrm{ln}\ r)^{q-1}, \quad \forall r>>1;
   \end{align}
   \\
    (b)  $(s,t)\in G_{2,b}=G_2\cap\{(s,t)|\ s<-1,t\geq p\}$ and there is 
  \begin{align}
       V(r)\leq C r^{\alpha}(\mathrm{ln}\ r)^{\alpha-1}, \quad \forall r>>1,
   \end{align}
   where $\alpha=\frac{ps+qt-(q-1)p}{s+t-p+1}$;
   \\
   (c)  $(s,t)\in G_{2,c}=G_2\cap\{(s,t)|\ s\geq -1,t<q\}$ and there is 
  \begin{align}
       V(r)\leq C r^{\alpha}(\mathrm{ln}\ r)^{\alpha-1}, \quad \forall r>>1,
   \end{align}
   where $\alpha=\frac{t}{t-p+1}(1+\frac{(q-p)(s+1)}{s+t-q+1})$.
 \end{theorem}

 \begin{theorem}
     For $(s,t)\in G_3$, (1.5) admits no nontrivial positive weak solution if one of the following assumptions holds:
     
     ~\\
    (a)  $(s,t)\in G_{3,a}=G_3\cap\{(s,t)|\ s<-1,t> q\}$ and there is 
  \begin{align}
       V(r)\leq C r^{\alpha}(\mathrm{ln}\ r)^{\alpha-1}, \quad \forall r>>1,
   \end{align}
   where $\alpha=\frac{qs+pt-(p-1)q}{s+t-q+1}$;
   \\
   (b)  $(s,t)\in G_{3,b}=G_3\cap\{(s,t)|\ s\geq  -1,t<p\}$ and there is 
  \begin{align}
       V(r)\leq C r^{\alpha}(\mathrm{ln}\ r)^{\alpha-1}, \quad \forall r>>1,
   \end{align}
   where $\alpha=\frac{t}{t-q+1}(1+\frac{(p-q)(s+1)}{s+t-p+1})$;
   \\
   (c)  $(s,t)\in G_{3,c}=G_3\cap\{(s,t)|\ s< -1,t\leq q\}$ and there is 
  \begin{align}
       V(r)\leq C r^{\frac{t}{t-q+1}}(\mathrm{ln}\ r)^{\frac{q-1}{t-q+1}}, \quad \forall r>>1.
   \end{align}
 \end{theorem}
    \begin{theorem}  For $(s,t)\in G_4$, (1.5) admits no nontrivial positive weak solution if one of the following assumptions holds:
    
        ~\\ (a)  $(s,t)\in G_{4,a}=G_4\cap \{t=q-1\}$ and there exists a constant $\alpha>0$ such that
        \begin{align}
            V(r)\leq Cr^\alpha,\quad \forall r>>1;
        \end{align}     
        (b)  $(s,t)\in G_{4,b}=G_4\cap \{t<q-1\}$ and there exists a constant $0<\kappa<\mathrm{min}\{\frac{t-p+1}{s+t-p+1},\frac{t-q+1}{s+t-q+1}\}$ such that
        \begin{align}
            V(r)\leq C e^{\kappa r\mathrm{ln}\ r},\quad \forall r>>1.
        \end{align} 
    \end{theorem}
    In the case of Euclidean space $\mathbb{R}^n$, we have the following corollary.
    \begin{corollary}
        Let $M=\mathbb{R}^n$ with standard metric, then (1.5) possesses no nontrivial positive solution if one of the following holds:
        \\
        (1) $(s,t)\in G_{1,a}\cap\{(s,t)|\ n\leq \frac{qs+t}{s+t-q+1}\}$;
        \\
        (2) $(s,t)\in (G_{1,b}\bigcup G_{2,a})\cap\{(s,t)|\ n\leq q\}$;
        \\
        (3) $(s,t)\in G_{2,b}\cap\{(s,t)|\ n\leq \frac{ps+qt-(q-1)p}{s+t-p+1}\}$;
        \\
        (4) $(s,t)\in G_{2,c}\cap\{(s,t)|\ n\leq \frac{t}{t-p+1}(1+\frac{(q-p)(s+1)}{s+t-q+1})\}$;
        \\
        (5) $(s,t)\in G_{3,a}\cap\{(s,t)|\ n\leq \frac{qs+pt-(p-1)q}{s+t-q+1}\}$;
        \\
        (6) $(s,t)\in G_{3,b}\cap\{(s,t)|\ n\leq \frac{t}{t-q+1}(1+\frac{(p-q)(s+1)}{s+t-p+1})\}$;
        \\
        (7) $(s,t)\in G_{3,c}\cap\{(s,t)|\ n\leq \frac{t}{t-q+1}\}$;
        \\
        (8) $(s,t)\in G_{4} $.
        \end{corollary}
        \begin{remark}
        ~\\
            (1) We point out that the condition $n\leq \frac{qs+t}{s+t-q+1} $ is equivalent to the condition (1.8) in \cite{BBF} when $s+t-q+1>0$. Hence the corollary 1.6 extends the Liouville results in \cite{BBF}. For example, the (1) and (2) in Corollary 1.6 improve the Theorem 1.8 in \cite{BBF}. However, since the proof relies heavily on Lemma 2.1, we cannot obtain the Liouville theorem for all $(s,t)\in \mathbb{R}^2$, such as $(s,t)\in \{ q-1\leq s+t\leq p-1\}$. 
            \\
            (2) Different from \cite{BBF}, the higher order term, that is $p-$Laplace operator, has a stronger influence on the structure of the analysis when $s<0$.
            \\
            (3) From the proof of the Theorem and the approach in \cite{VGM}, we can obtain Liouville type results for a wider class of inequalities of the type
            \begin{align*}
                   \mathrm{div}(|\nabla u|^{q-2}f(|\nabla u|)\nabla u)+u^s|\nabla u|^t\leq 0,\quad in\  M.
               \end{align*}
               Here $f$ satisfies the condition that there exists constants $b\geq a>0,K\geq 0$ and $v> 0$ such that
               \begin{align*}
                   at^{v}\leq f(t) \leq K+bt^v, \quad \forall t\geq 0.
               \end{align*}
        \end{remark}
       Finally, we present some examples to illustrate sharpness in certain cases.
       \begin{theorem}
           There exists a geodesically complete noncompact Riemannian manifold $M$ such that the inequality (1.5) has a nontrivial positive solution if the volume $V(r)$ satisfies one of the following conditions for sufficiently large $r$: 
           
           ~\\
           (1) $(s,t)\in G_{1,a}$ and $V(r)\leq C r^{\frac{qs+t}{s+t-q+1}}(\mathrm{ln}\ r)^{\frac{q-1}{s+t-q+1}+\epsilon} $ for  $\epsilon>0$;
   
   ~\\
   (2) $(s,t)\in G_{1,b}\bigcup G_{2,a}$ and $V(r)\leq C r^{q}(\mathrm{ln}\ r)^{q-1+\epsilon}$ for $\epsilon>0$; 
  
   ~\\
   (3) $(s,t)\in G_{3,c}\cap \{t>p-1\}$ and $V(r)\leq C r^{\frac{t}{t-q+1}}(\mathrm{ln}\ r)^{\frac{q-1}{t-q+1}+\epsilon}$  for $\epsilon>0$;
   
   ~\\ (4) $(s,t)\in G_{4,a}$ and $V(r)\leq Ce^{\lambda r} $ for $\lambda>0$;   

   ~\\
        (5) $(s,t)\in G_{4,b}$ and $V(r)\leq C e^{\lambda r^l\mathrm{ln}\ r}$ for $l>2(q-t-1)+1,\lambda>0$.
       \end{theorem}
       \par
           The rest of the paper is organized as follows. In Section 2, we prove some preliminary results, which will be used in the proof. In Section 3, we prove the main results. In section 4, we give some examples.

      \section{Preliminary}
      In order to prove Theorem 1.2-Theorem 1.5, we first present a useful lemma. 
      \begin{lemma}
          Assume that $ s+t\notin \{ p-1,q-1\}$ and $u$ is a no nontrivial positive weak solution of (1.5). Then there exist a positive pair $(a,b)$ and a constant $C>0$ such that for any $0\leq \varphi \leq 1,\varphi\in W^{1,p}_{loc}(M)$ with compact support, one has  
          \begin{align}
              &\int_M u^{s-a}|\nabla u|^t\varphi^bd\mu 
              \nonumber\\
              \leq & C(2b)^{\frac{ps+t+a(t-p)}{s+t-a}}a^{-\frac{(p-1)s+a(t-p+1)}{s+t-a}}\left(\int_M|\nabla \varphi|^{\frac{ps+t+a(t-p)}{s+t-p+1}}d\mu \right)^{\frac{s+t-p+1}{s+t-a}} \nonumber\\
              & \cdot \left(\int_{supp |\nabla \varphi|} u^{s-a}|\nabla u|^t\varphi^bd\mu \right)^{\frac{p-a-1}{s+t-a}}
              \nonumber\\
              & +C(2b)^{\frac{qs+t+a(t-q)}{s+t-a}}a^{-\frac{(q-1)s+a(t-q+1)}{s+t-a}}\left(\int_M|\nabla \varphi|^{\frac{qs+t+a(t-q)}{s+t-q+1}}d\mu \right)^{\frac{s+t-q+1}{s+t-a}} \nonumber\\
              & \cdot \left(\int_{supp |\nabla \varphi|} u^{s-a}|\nabla u|^t\varphi^bd\mu \right)^{\frac{q-a-1}{s+t-a}}
          \end{align}
          and 
          \begin{align}
              &\int_M u^{s-a}|\nabla u|^t\varphi^bd\mu 
              \nonumber\\
              \leq& C^{\frac{s+t-a}{s+t-p+1}}(2b)^{\frac{ps+t+a(t-p)}{s+t-p+1}}a^{-\frac{(p-1)s+a(t-p+1)}{s+t-p+1}}\int_M|\nabla \varphi|^{\frac{ps+t+a(t-p)}{s+t-p+1}}d\mu 
              \nonumber\\
              &+ C^{\frac{s+t-a}{s+t-q+1}}(2b)^{\frac{qs+t+a(t-q)}{s+t-q+1}}a^{-\frac{(q-1)s+a(t-q+1)}{s+t-q+1}}\int_M|\nabla \varphi|^{\frac{qs+t+a(t-q)}{s+t-q+1}}d\mu 
          \end{align}
          if $a,b$ satisfy that
          \begin{align}  
    \begin{cases}
        \frac{ps+t+a(t-p)}{s+t-a}>1, \quad  \frac{s+t-a}{p-a-1}>1; \\
           \\
           \frac{qs+t+a(t-q)}{s+t-a}>1,\quad  \frac{s+t-a}{q-a-1}>1;\\
           \\
           b>\mathrm{max}\{\frac{ps+t+a(t-p)}{s+t-a}, \frac{qs+t+a(t-q)}{s+t-a}\}.
    \end{cases}   
\end{align}
Here the constant $C$ does not depend on $a$.
      \end{lemma}
      \begin{proof}
        Taking $\psi=u^{-a}\varphi^b$ in (1.6), then we have
               \begin{align}
                 &\int_Mu^{s-a}|\nabla u|^t\varphi^bd\mu +a\int_Mu^{-a-1}|\nabla u|^p\varphi^bd\mu +a\int_Mu^{-a-1}|\nabla u|^q\varphi^bd\mu
                 \nonumber\\
                 \leq &b\int_Mu^{-a}|\nabla u|^{p-2}\varphi^{b-1}\langle \nabla u,\nabla\varphi\rangle d\mu+b\int_Mu^{-a}|\nabla u|^{q-2}\varphi^{b-1}\langle \nabla u,\nabla\varphi\rangle d\mu.
               \end{align}
               First, we estimate the right side of (2.4). Let 
               \begin{align*}
                   \alpha_p=\frac{ps+t+a(t-p)}{(p-1)s+a(t-p+1)}>1,\quad \beta_p=\frac{ps+t+a(t-p)}{s+t-a}>1.
               \end{align*}
               From Young's inequality with the pair $(\frac{1}{\alpha_p},\frac{1}{\beta_p})$, we obtain
               \begin{align}
                   &b\int_Mu^{-a}|\nabla u|^{p-2}\varphi^{b-1}\langle \nabla u,\nabla\varphi\rangle d\mu\leq b\int_Mu^{-a}|\nabla u|^{p-1}\varphi^{b-1}|\nabla\varphi| d\mu
                   \nonumber\\
                   = & \int_M \left(\frac{a}{2}\right)^{\frac{1}{\alpha_p}}u^{\frac{-a-1}{\alpha_p}}|\nabla u|^{\frac{p}{\alpha_p}} \varphi^{\frac{b}{\alpha_p}}\cdot b \left(\frac{a}{2}\right)^{-\frac{1}{\alpha_p}}u^{-a+\frac{a+1}{\alpha_p}}|\nabla u|^{p-1-\frac{p}{\alpha_p}} \varphi^{b-1-\frac{b}{\alpha_p}}|\nabla \varphi|d\mu
                   \nonumber\\
                   \leq & \frac{a}{2}\int_Mu^{-a-1}|\nabla u|^p\varphi^bd\mu
                   \nonumber\\
                   &+b^{\beta_p}\left(\frac{a}{2}\right)^{-\frac{\beta_p}{\alpha_p}}\int_M u^{-a\beta_p+\frac{a+1}{\alpha_p}\beta_p}|\nabla u|^{(p-1)\beta_p-\frac{p}{\alpha_p}\beta_p} \varphi^{(b-1)\beta_p-\frac{b}{\alpha_p}\beta_p}|\nabla \varphi|^{\beta_p}d\mu
                   \nonumber\\
                   =&\frac{a}{2}\int_Mu^{-a-1}|\nabla u|^p\varphi^bd\mu+b^{\beta_p}\left(\frac{a}{2}\right)^{-\beta_p+1}\int_M u^{\beta_p-a-1}|\nabla u|^{p-\beta_p} \varphi^{b-\beta_p}|\nabla \varphi|^{\beta_p}d\mu.
               \end{align}
               Similarly, letting 
               \begin{align*}
                   \alpha_q=\frac{qs+t+a(t-q)}{(q-1)s+a(t-q+1)}>1,\quad \beta_q=\frac{qs+t+a(t-q)}{s+t-a}>1,
               \end{align*}
               we derive that 
               \begin{align}
                   &b\int_Mu^{-a}|\nabla u|^{q-2}\varphi^{b-1}\langle \nabla u,\nabla\varphi\rangle d\mu
                   \nonumber\\
                   \leq &\frac{a}{2}\int_Mu^{-a-1}|\nabla u|^q\varphi^bd\mu+b^{\beta_q}\left(\frac{a}{2}\right)^{-\beta_q+1}\int_M u^{\beta_q-a-1}|\nabla u|^{q-\beta_q} \varphi^{b-\beta_q}|\nabla \varphi|^{\beta_q}d\mu.
               \end{align}
               Putting (2.5) and (2.6) into (2.4), we have
               \begin{align}
                 &\int_Mu^{s-a}|\nabla u|^t\varphi^bd\mu +\frac{a}{2}\int_Mu^{-a-1}|\nabla u|^p\varphi^bd\mu +\frac{a}{2}\int_Mu^{-a-1}|\nabla u|^q\varphi^bd\mu
                 \nonumber\\
                 \leq &b^{\beta_p}\left(\frac{a}{2}\right)^{-\beta_p+1}\int_M u^{\beta_p-a-1}|\nabla u|^{p-\beta_p} \varphi^{b-\beta_p}|\nabla \varphi|^{\beta_p}d\mu
                 \nonumber\\
                 &+b^{\beta_q}\left(\frac{a}{2}\right)^{-\beta_q+1}\int_M u^{\beta_q-a-1}|\nabla u|^{q-\beta_q} \varphi^{b-\beta_q}|\nabla \varphi|^{\beta_q}d\mu.
               \end{align}
               Set 
               \begin{align*}
                   \gamma_p=\frac{s+t-a}{p-a-1}>1, \quad \rho_p=\frac{s+t-a}{s+t-p+1}>1.
               \end{align*}
               Since $\frac{p-\beta_p}{\beta_p-a-1}=\frac{t}{s-a}$, applying H$\ddot{\mathrm{o}}$lder's inequality with the pair $(\frac{1}{\gamma_p},\frac{1}{\rho_p})$, we derive
               \begin{align}
                   &\int_M u^{\beta_p-a-1}|\nabla u|^{p-\beta_p} \varphi^{b-\beta_p}|\nabla \varphi|^{\beta_p}d\mu
                   \nonumber\\
                   =&\int_M u^{\beta_p-a-1}|\nabla u|^{p-\beta_p} \varphi^{\frac{b}{\gamma_p}}\cdot \varphi^{b-\beta_p-\frac{b}{\gamma_p}}|\nabla \varphi|^{\beta_p}d\mu
                   \nonumber\\
                   \leq & \left(\int_{supp|\nabla\varphi|}u^{s-a}|\nabla u|^t\varphi^bd\mu \right)^{\frac{1}{\gamma_p}}\cdot \left(
                    \int_M \varphi^{b-\beta_p\rho_p}|\nabla\varphi|^{\beta_p\rho_p}d\mu\right)^{\frac{1}{\rho_p}}.
               \end{align}
               Similarly, letting 
               \begin{align*}
                   \gamma_q=\frac{s+t-a}{q-a-1}>1, \quad \rho_q=\frac{s+t-a}{s+t-q+1}>1,
               \end{align*}
               we obtain
               \begin{align}
                   &\int_M u^{\beta_q-a-1}|\nabla u|^{q-\beta_q} \varphi^{b-\beta_q}|\nabla \varphi|^{\beta_q}d\mu
                   \nonumber\\
                   \leq & \left(\int_{supp|\nabla\varphi|}u^{s-a}|\nabla u|^t\varphi^bd\mu \right)^{\frac{1}{\gamma_q}}\cdot \left(
                    \int_M \varphi^{b-\beta_q\rho_q}|\nabla\varphi|^{\beta_q\rho_q}d\mu\right)^{\frac{1}{\rho_q}}.
               \end{align}
               Note that $0\leq \varphi \leq 1$ and $b>\mathrm{max}\{\frac{ps+t+a(t-p)}{s+t-a}, \frac{qs+t+a(t-q)}{s+t-a}\} $. Combining (2.8), (2.9) with (2.7), we have
               \begin{align*}
                   &\int_M u^{s-a}|\nabla u|^t\varphi^bd\mu
                   \nonumber\\
                   \leq &b^{\beta_p}\left(\frac{a}{2}\right)^{-\beta_p+1}\left(\int_{supp|\nabla\varphi|}u^{s-a}|\nabla u|^t\varphi^bd\mu \right)^{\frac{1}{\gamma_p}}\cdot \left(
                    \int_M |\nabla\varphi|^{\beta_p\rho_p}d\mu\right)^{\frac{1}{\rho_p}}
                    \\
                    &+ b^{\beta_q}\left(\frac{a}{2}\right)^{-\beta_q+1}\left(\int_{supp|\nabla\varphi|}u^{s-a}|\nabla u|^t\varphi^bd\mu \right)^{\frac{1}{\gamma_q}}\cdot \left(
                    \int_M |\nabla\varphi|^{\beta_q\rho_q}d\mu\right)^{\frac{1}{\rho_q}}.
               \end{align*}
               Then (2.1) follows from the definition of $ \beta_p,\gamma_p,\rho_p,\beta_q,\rho_q,\gamma_q$. Next we prove (2.2). Since $u^s|\nabla u|^t\in L_{loc}^{1}(M) $ and $u^{-1}\in L^{\infty}_{loc}(M) $, we have
               \begin{align*}
                   \int_M u^{s-a}|\nabla u|^t\varphi^bd\mu<\infty.
               \end{align*}
               Noting that if $x<\infty$ satisfies 
               \begin{align*}
                   x^c-c_1x^d-c_2\leq 0, \quad c>d>0,c_1,c_2>0 ,
               \end{align*}
               then 
               \begin{align*}
                   x\leq (2c_1)^{\frac{1}{c-d}}+(2c_2)^{\frac{1}{c}}.
               \end{align*}
               Letting $x=\int_M u^{s-a}|\nabla u|^t\varphi^bd\mu$ and combining with (2.1), we obtain (2.2). 
      \end{proof}
      In the remaining part of this section, we assume (2.3) holds and discuss the choice of $a$. Rearranging (2.3), we have
      \begin{align}  
          \frac{s+t-a}{p-a-1}>1,\quad 
   \frac{s+t-a}{q-a-1}>1
\end{align}
and 
\begin{align}  
        \frac{ps+t+a(t-p)}{s+t-a}>1,\quad 
           \frac{qs+t+a(t-q)}{s+t-a}>1.
\end{align}
    From (2.10), we have $ (p-a-1)(q-a-1)>0$. Hence, the discussion can be divided into two parts. 
    \\
    (1) $0<a<q-1$. In this case, we have $s+t>p-1$ and 
    \begin{align}
        (p-1)s+a(t-p+1)>0,\quad (q-1)s+a(t-q+1)>0.
    \end{align}
    Then (2.12) holds if $s,t,a$ satisfy one of the following conditions:
    \par
    (1.a)  $t> p-1$, $s\geq 0$ and $0<a<q-1$;
    \par
    (1.b)  $t>p-1$,$s<0$ and
    \begin{align*}
        \frac{s}{q-1}+\frac{t}{p-1}>1,\quad \frac{(1-p)s}{t-p+1}<a<q-1;
    \end{align*}
    \par
    (1.c) $q-1\leq t\leq p-1 $ and 
    \begin{align*}
        s>0,\quad 0<a<\mathrm{min}\left\{\frac{(p-1)s}{p-t-1},q-1\right\}=q-1;
    \end{align*}
    \par
    (1.d) $ t< q-1 $ and
    \begin{align*}
        s>0,\quad 0<a<\mathrm{min}\left\{\frac{(p-1)s}{p-t-1},\frac{(q-1)s}{q-t-1},q-1\right\}=q-1.
    \end{align*}
    (2) $a>p-1$. In this case, we have $s+t<q-1$ and 
    \begin{align}
        (p-1)s+a(t-p+1)<0,\quad (q-1)s+a(t-q+1)<0.
    \end{align}
    Then (2.13) holds if $s,t,a$ satisfy one of the following conditions:
    \par
    (2.a) $t>p-1,s<0$ and 
    \begin{align*}
        \frac{s}{p-1}+\frac{t}{q-1}<1, \quad  p-1<a<\mathrm{min}\left\{\frac{(1-p)s}{t-p+1},\frac{(1-q)s}{t-q+1}\right\}=\frac{(1-q)s}{t-q+1};
    \end{align*}
    \par
    (2.b) $q-1\leq t\leq p-1,s<0$ and
    \begin{align*}
        \frac{s}{p-1}+\frac{t}{q-1}<1, \quad  p-1<a<\frac{(1-q)s}{t-q+1};
    \end{align*}
    \par
    (2.c) $t\leq q-1, s< 0$ and $a>p-1$;
    \par
    (2.d) $t<q-1,s\geq 0$ and 
    \begin{align*}
        a>\mathrm{max}\left\{p-1,\frac{(q-1)s}{q-t-1},\frac{(p-1)s}{p-t-1}\right\}=p-1.
    \end{align*}
    Based on the above discussions, we make the following admissible choice of $a$:
    \begin{align} 
    \begin{cases}
        G_1=\{(s,t)|\  s\geq 0,t>p-s-1 \}; \\
           \\
           G_2=\{(s,t) |\ s<0,\frac{s}{q-1}+\frac{t}{p-1}>1\}; \\
           \\
           G_3=\{(s,t)|\ t>q-1,\frac{s}{p-1}+\frac{t}{q-1}<1\} ;\\
           \\
           G_4=\{(s,t)|\ t\leq q-1,s<q-t-1 \},
    \end{cases} 
    \Longrightarrow 
    \begin{cases}
        0<a<q-1; \\
           \\
           \frac{(1-p)s}{t-p+1}<a<q-1; \\
           \\
            p-1<a<\frac{(1-q)s}{t-q+1};\\
           \\
           a>p-1 .
    \end{cases}
\end{align}

    \section{Proof of the main results}
    In this section, we assume that $u$ is a nontrivial positive weak solution of (1.6) and Lemma 2.1 holds. Before the proof, we determine the test function in (1.6). First, we construct the following cut-off function which was introduced in \cite{SXX}. Define $\eta_k=h(\frac{r(x)}{2^k})$, where $r(x)=d(x,x_0)$ for some fixed $x_0\in M$ and $h$ is a smooth function satisfying
    \begin{align*}  
    \begin{cases}
        0\leq h(t)\leq 1, &    t\in [0,\infty) ,
           \\
           h(t)=1,  & t\in [0,1),
           \\
           h(t)=0, & t\in [2,\infty).
    \end{cases} 
\end{align*}
 For sufficiently large $i\in \mathbb{N}$, letting $\varphi_i(x)=i^{-1}\sum\limits_{k=i+1}^{2i}\eta_k(x)$, then it is easy to see that
  \begin{align*}
    \varphi_i(x) &= 
    \begin{cases}
        1, & \text{if } x \in B_{2^{i+1}}; \\ 
        0,  & \text{if } x\notin B_{2^{2i+1}}^c.
    \end{cases} 
\end{align*}
Noting that the support of $\nabla\eta_k$ is different, we have
\begin{align*}
    |\nabla\varphi_i|^\theta\leq C i^{-\theta}\sum\limits_{k=i+1}^{2i}2^{-k\theta}\chi_{2^k\leq r(\cdot)\leq 2^{k+1}}(x),
\end{align*}
where $\chi$ is the characteristic function. For convenience, we collect here some notations that we are going to use in the proof appearing below. We define 
\begin{align}
    &J= \int_M u^{s-a}|\nabla u|^t\varphi_i^bd\mu,
    \\
    &I_p=(2b)^{\frac{ps+t+a(t-p)}{s+t-p+1}}a^{-\frac{(p-1)s+a(t-p+1)}{s+t-p+1}}\int_M|\nabla \varphi_i|^{\frac{ps+t+a(t-p)}{s+t-p+1}}d\mu,
    \\
    &I_q=(2b)^{\frac{qs+t+a(t-q)}{s+t-q+1}}a^{-\frac{(q-1)s+a(t-q+1)}{s+t-q+1}}\int_M|\nabla \varphi_i|^{\frac{qs+t+a(t-q)}{s+t-q+1}}d\mu.
\end{align}
Then (2.2) implies that 
\begin{align}
    J\leq C^{\frac{s+t-a}{s+t-p+1}}I_p+C^{\frac{s+t-a}{s+t-q+1}}I_q.
\end{align}
The key in the proof is to estimate $J,I_p,I_q$. Now we begin the proof.

~~\\
         $\mathit{Proof\ of\ Theorem\ 1.2.}$ Since $(s,t)\in G_1$ and (2.14), we have $0<a<q-1.$
         
         ~\\
         (a) $(s,t)\in G_{1,a}$. We choose $a=i^{-1}$ and a fixed $b>\mathrm{max}\{\frac{ps+t+a(t-p)}{s+t-a}, \frac{qs+t+a(t-q)}{s+t-a}\}$, which yields that
         \begin{align}
             C^{\frac{s+t-a}{s+t-p+1}}\leq C,\quad (2b)^{\frac{ps+t+a(t-p)}{s+t-p+1}}\leq C. \nonumber
         \end{align}
         Hence we find 
         \begin{align}
             I_q\leq &Ci^{\frac{(q-1)s+a(t-q+1)}{s+t-q+1}}\int_M|\nabla \varphi_i|^{\frac{qs+t+a(t-q)}{s+t-q+1}}d\mu
             \nonumber\\
             \leq & C i^{\frac{(q-1)s+a(t-q+1)}{s+t-q+1}}i^{-\frac{qs+t+a(t-q)}{s+t-q+1} }\left(\sum\limits_{k=i+1}^{2i}\int_{B_{2^{k+1}}\setminus B_{2^{k}}}2^{-k\frac{qs+t+a(t-q)}{s+t-q+1}}d\mu  \right)
             \nonumber\\
             \leq & C i^{-\frac{s+t-a}{s+t-q+1}}\left(\sum\limits_{k=i+1}^{2i}2^{-k\frac{qs+t+a(t-q)}{s+t-q+1}}V(2^{k+1})  \right)
             \nonumber\\
             \leq & C i^{-\frac{s+t-a}{s+t-q+1}}\left(\sum\limits_{k=i+1}^{2i}2^{k( \frac{qs+t}{s+t-q+1}-\frac{qs+t+a(t-q)}{s+t-q+1})}k^{\frac{q-1}{s+t-q+1}}  \right)
             \nonumber\\
             \leq & C i^{ \frac{i^{-1}}{s+t-q+1}}\leq C,
         \end{align}
         where we use (1.7) in the fourth inequality. Since $(s,t)\in G_{1,a}$, we have
         \begin{align*}
             \frac{qs+t}{s+t-q+1}\leq \frac{ps+t}{s+t-p+1}, \quad \frac{q-1}{s+t-q+1}\leq \frac{p-1}{s+t-p+1}.
         \end{align*}
         By the same argument, we obtain
        \begin{align}
            I_p
             \leq & C i^{-\frac{s+t-a}{s+t-p+1}}\left(\sum\limits_{k=i+1}^{2i}2^{k( \frac{ps+t}{s+t-p+1}-\frac{ps+t+a(t-p)}{s+t-p+1})}k^{\frac{p-1}{s+t-p+1}}  \right)
             \nonumber\\
             \leq & C i^{ \frac{i^{-1}}{s+t-p+1}}\leq C.
        \end{align}
        Putting (3.5) and (3.6) into (3.4), we derive
        \begin{align*}
            \int_{B_{2^{i+1}}}u^{s-i^{-1}}|\nabla u|^td\mu\leq J\leq C .
        \end{align*}
        Letting $i\xrightarrow{}\infty$, we have
        \begin{align*}
            \int_{M}u^{s}|\nabla u|^td\mu\leq C.
        \end{align*}
        Using (2.1) and repeating the same procedure, we have
        \begin{align*}
          &  \int_{B_{2^{i+1}}}u^{s-i^{-1}}|\nabla u|^td\mu 
          \\
          \leq& C \left(\int_{M\setminus B_{2^{i+1}}}u^{s-i^{-1}}|\nabla u|^td\mu\right)^{\frac{s+t-p+1}{s+t-i^{-1}}}+C\left(\int_{M\setminus B_{2^{i+1}}}u^{s-i^{-1}}|\nabla u|^td\mu\right)^{\frac{s+t-q+1}{s+t-i^{-1}}}.
        \end{align*}
        Letting $i\xrightarrow{}\infty$ again, we obtain 
        \begin{align*}
            \int_{M}u^{s}|\nabla u|^td\mu=0
        \end{align*}
        which contradicts that $u$ is nontrivial.

        ~\\
        (b) $(s,t)\in G_{1,b}$. In this case, we choose $a=q-1-i^{-1}$ and a fixed $b>\mathrm{max}\{\frac{ps+t+a(t-p)}{s+t-a}, \frac{qs+t+a(t-q)}{s+t-a}\}$, which yields that
         \begin{align}
             \mathrm{max}\{C^{\frac{s+t-a}{s+t-p+1}},\ (2b)^{\frac{ps+t+a(t-p)}{s+t-p+1}},\ a^{-\frac{(p-1)s+a(t-p+1)}{s+t-p+1}}, \ a^{-\frac{(q-1)s+a(t-q+1)}{s+t-q+1}}\}\leq C. 
         \end{align}
         As above, we have
        \begin{align}
            I_q \leq &C\int_M|\nabla \varphi|^{\frac{qs+t+a(t-q)}{s+t-q+1}}d\mu
             \nonumber\\
             \leq & C i^{-\frac{qs+t+a(t-q)}{s+t-q+1} }\left(\sum\limits_{k=i+1}^{2i}\int_{B_{2^{k+1}}\setminus B_{2^{k}}}2^{-k\frac{qs+t+a(t-q)}{s+t-q+1}}d\mu  \right)
             \nonumber\\
             \leq &C i^{-\frac{qs+t+a(t-q)}{s+t-q+1} } \left(\sum\limits_{k=i+1}^{2i}2^{-k\frac{qs+t+a(t-q)}{s+t-q+1}}V(2^{k+1})  \right)
             \nonumber\\
             \leq &C i^{-\frac{qs+t+a(t-q)}{s+t-q+1} } \left(\sum\limits_{k=i+1}^{2i}2^{k( q-\frac{qs+t+(q-1)(t-q)-i^{-1}(t-q)}{s+t-q+1})}k^{\frac{q-1}{s+t-q+1}}  \right)
             \nonumber\\
             \leq &Ci^{\frac{i^{-1}(t-q)}{s+t-q+1}}\leq C. 
        \end{align}
        It is easy to see that $I_p\leq C$ since $q\leq p$, which implies that
        $J\leq C$ and 
        \begin{align*}
            \int_{B_{2^{i+1}}}u^{s-q+1+i^{-1}}|\nabla u|^td\mu \leq C.
        \end{align*}
        In particular, letting $i\xrightarrow{} \infty$, we have
        \begin{align}
            \int_{M}u^{s-q+1}|\nabla u|^td\mu \leq C.
        \end{align}
        Using (2.1) and repeating the same procedure, we obtain
        \begin{align*}
          &  \int_{B_{2^{i+1}}}u^{s-q+1+i^{-1}}|\nabla u|^td\mu 
          \\
          \leq& C \left(\int_{M\setminus B_{2^{i+1}}}u^{s-q+1+i^{-1}}|\nabla u|^td\mu\right)^{\frac{s+t-p+1}{s+t-i^{-1}}}+C\left(\int_{M\setminus B_{2^{i+1}}}u^{s-q+1+i^{-1}}|\nabla u|^td\mu\right)^{\frac{s+t-q+1}{s+t-i^{-1}}}.
        \end{align*}
        Letting $i\xrightarrow{}\infty$ again, we achieve 
        \begin{align*}
            \int_{M}u^{s-q+1}|\nabla u|^td\mu=0
        \end{align*}
        which contradicts that $u$ is nontrivial. This completes the proof. \qed

        ~\\
        $\mathit{Proof\ of\ Theorem\ 1.3.}$ Since $(s,t)\in G_2$ and (2.14), we have
        \begin{align*}
            \frac{(1-p)s}{t-p+1}<a<q-1.
        \end{align*}
        We assume that $V(r)\leq Cr^\alpha (\mathrm{ln}\ r)^{\beta}$ for $r$ large enough. As in the proof of Theorem 1.2, we only need to show that $I_p\leq C$ and $I_q\leq C$. Noting that $\frac{(1-p)s}{t-p+1}<a<q-1$, then $a$ can take the value of $\frac{(1-p)s}{t-p+1}+i^{-1} $ or $q-1-i^{-1}$. Hence we can choose a fixed $b>\mathrm{max}\{\frac{ps+t+a(t-p)}{s+t-a}, \frac{qs+t+a(t-q)}{s+t-a}\} $ and (3.7) holds. Then, for $i$ large enough, we have
        \begin{align}
             I_q \leq &C\int_M|\nabla \varphi_i|^{\frac{qs+t+a(t-q)}{s+t-q+1}}d\mu
             \nonumber\\
             \leq & C i^{-\frac{qs+t+a(t-q)}{s+t-q+1} }\left(\sum\limits_{k=i+1}^{2i}\int_{B_{2^{k+1}}\setminus B_{2^{k}}}2^{-k\frac{qs+t+a(t-q)}{s+t-q+1}}d\mu  \right)
             \nonumber\\
             \leq &C i^{-\frac{qs+t+a(t-q)}{s+t-q+1} } \left(\sum\limits_{k=i+1}^{2i}2^{-k\frac{qs+t+a(t-q)}{s+t-q+1}}V(2^{k+1})  \right)
             \nonumber\\
             \leq & C i^{-\frac{qs+t+a(t-q)}{s+t-q+1} } \left(\sum\limits_{k=i+1}^{2i}2^{k(\alpha-\frac{qs+t+a(t-q)}{s+t-q+1})}k^\beta  \right)
        \end{align}
        and 
        \begin{align}
            I_p\leq C i^{-\frac{ps+t+a(t-p)}{s+t-p+1} } \left(\sum\limits_{k=i+1}^{2i}2^{k(\alpha-\frac{ps+t+a(t-p)}{s+t-p+1})}k^\beta  \right).
        \end{align}
        As in the proof of Theorem 1.2, we have
        \\
        (1) In the case $a=\frac{(1-p)s}{t-p+1}+i^{-1}$, $I_p\leq C$ and $I_q\leq C$ $\beta=\alpha-1$ and
        \begin{align}
             \alpha\leq \mathrm{min}\{\frac{t}{t-p+1} , \frac{qs+t+\frac{(1-p)(t-q)s}{t-p+1}}{s+t-q+1}\}.
        \end{align}
        (2) In the case $a=q-1-i^{-1}$, $I_p\leq C$ and $I_q\leq C$ if $\beta=\alpha-1$ and 
        \begin{align}
            \alpha\leq \mathrm{min}\{q , \frac{ps+t+(t-p)(q-1)}{s+t-p+1}\}.
        \end{align}
        Consequently, from (3.12) and (3.13), we have $I_p\leq C$ and $I_q\leq C$ if $\beta=\alpha-1$ and 
        \begin{align}
            \alpha=\mathrm{max}\{\mathrm{min}\{\frac{t}{t-p+1} , \frac{qs+t+\frac{(1-p)(t-q)s}{t-p+1}}{s+t-q+1}\}, \mathrm{min}\{q , \frac{ps+t+(t-p)(q-1)}{s+t-p+1}\}\}.
        \end{align}
        A direct computation shows that
        \begin{align*}
            \frac{qs+t+\frac{(1-p)(t-q)s}{t-p+1}}{s+t-q+1}=\frac{t}{t-p+1}(1+\frac{(q-p)(s+1)}{s+t-q+1})
        \end{align*}
        and 
        \begin{align*}
            \frac{ps+t+(t-p)(q-1)}{s+t-p+1}=q+\frac{(p-q)(s+1)}{s+t-p+1}.
        \end{align*}
        Hence, we can choose 
        \begin{align}
            \alpha=\mathrm{max}\{ \frac{qs+t+\frac{(1-p)(t-q)s}{t-p+1}}{s+t-q+1}, q \}
        \end{align}
        if $s\geq -1$ and
        \begin{align}
            \alpha=\mathrm{max}\{\frac{t}{t-p+1} ,  \frac{ps+t+(t-p)(q-1)}{s+t-p+1}\}
        \end{align}
        if $s<-1$. Note that 
        \begin{align}
            q-\frac{qs+t+\frac{(1-p)(t-q)s}{t-p+1}}{s+t-q+1}=\frac{(t-q)(p-1)(q-1)}{(s+t-q+1)(t-p+1)}(\frac{s}{q-1}+\frac{t}{p-1}-1)
        \end{align}
        and
        \begin{align}
            & \frac{ps+t+(t-p)(q-1)}{s+t-p+1}-\frac{t}{t-p+1}
            \nonumber\\
            =&\frac{(t-p)(p-1)(q-1)}{(s+t-p+1)(t-p+1)}(\frac{s}{q-1}+\frac{t}{p-1}-1).
        \end{align}
        From (3.14)-(3.18), we conclude that $I_p$ and $I_q$ are uniformly bounded if one of the following holds:
        \\
        (a) $(s,t)\in G_{2,a}=G_2\cap\{(s,t)|\ s\geq -1,t\geq q\}$ and $\alpha=q$;
        \\
        (b) $(s,t)\in G_{2,b}=G_2\cap\{(s,t)|\ s<-1,t\geq p\}$ and $\alpha=\frac{ps+t+(t-p)(q-1)}{s+t-p+1}$;
        \\
        (c) $(s,t)\in G_{2,c}=G_2\cap\{(s,t)|\ s\geq -1,t<q\}$ and $\alpha=\frac{t}{t-p+1}(1+\frac{(q-p)(s+1)}{s+t-q+1})$;
        \\
        (d) $(s,t)\in G_{2,d}=G_2\cap\{(s,t)|\ s< -1,t<p\}$ and $\alpha=\frac{t}{t-p+1}$. It follows from the calculation that the set $G_{2,d}$ is empty.
        \par
        Finally, by the same argument in the proof of Theorem 1.2, we have
        \begin{align*}
            \int_Mu^{s-a_\infty
            }|\nabla u|^td\mu =0,
        \end{align*}
        where $a_\infty=\lim\limits_{i\xrightarrow{}\infty}a$. This contradicts that $u$ is nontrivial. \qed

        ~\\
        $\mathit{Proof\ of\ Theorem\ 1.4.}$  The proof is similar to the previous one, so we only provide a brief sketch here. Since $(s,t)\in G_3$ and (2.14), we have 
        \begin{align*}
            p-1<a<\frac{(1-q)s}{t-q+1}.
        \end{align*}
        Thus, $a$ can take the value of $p-1+i^{-1} $ or $\frac{(1-q)s}{t-q+1}-i^{-1} $. We assume that $V(r)\leq Cr^\alpha (\mathrm{ln}\ r)^{\beta}$ for $r$ large enough. Similarly, $I_p$ and $I_q$ are uniformly bounded if $\beta=\alpha-1$ and
        \begin{align}
            \alpha=\mathrm{max}\{\mathrm{min}\{\frac{t}{t-q+1} , \frac{ps+t+\frac{(1-q)(t-p)s}{t-q+1}}{s+t-p+1}\}, \mathrm{min}\{p , \frac{qs+t+(t-q)(p-1)}{s+t-q+1}\}\}.
        \end{align}
       As in the proof of Theorem 1.3, we consider the following four cases:
       \\
       (1) $(s,t)\in G_3\cap\{(s,t)|\ s\geq -1,t\geq p\}=\emptyset$;
       \\
        (2)  $(s,t)\in G_{3,a}=G_3\cap\{(s,t)|\ s<-1,t\geq q\}$ with $\alpha=\frac{qs+pt-(p-1)q}{s+t-q+1}$;
   \\
     (3)  $(s,t)\in G_{3,b}=G_3\cap\{(s,t)|\ s\geq  -1,t<p\}$ with $\alpha=\frac{t}{t-q+1}(1+\frac{(p-q)(s+1)}{s+t-p+1})$;
   \\
   (4)  $(s,t)\in G_{3,c}=G_3\cap\{(s,t)|\ s< -1,t<q\}$ with $\alpha=\frac{t}{t-q+1}$.
   \\
   It follows that 
   \begin{align*}
            \int_Mu^{s-a_\infty
            }|\nabla u|^td\mu =0,
        \end{align*}
        where $a_\infty=\lim\limits_{i\xrightarrow{}\infty}a$. This leads to a contradiction.
        \qed
 
 ~\\
        $\mathit{Proof\ of\ Theorem\ 1.5.}$ Since $(s,t)\in G_4$, we have $a>p-1$. 
        \\
        (a) $(s,t)\in G_{4,a}=\{(s,t)|\ s<0,t=q-1\}$. In this case, we choose 
        \begin{align}
            a=l+i^{-1}
        \end{align}
        for $l$ large enough and a fixed $b>0$ depends on $l$. Then for large $i$, since $V(r)\leq Cr^\alpha  $ for $r>>1$, we obtain 
        \begin{align}
             I_q \leq &C\int_M|\nabla \varphi|^{\frac{qs+t+a(t-q)}{s+t-q+1}}d\mu
             \nonumber\\
             \leq & C i^{-\frac{qs+t+a(t-q)}{s+t-q+1} }\left(\sum\limits_{k=i+1}^{2i}\int_{B_{2^{k+1}}\setminus B_{2^{k}}}2^{-k\frac{qs+t+a(t-q)}{s+t-q+1}}d\mu  \right)
             \nonumber\\
             \leq & C i^{-\frac{qs+t+a(t-q)}{s+t-q+1} } \left(\sum\limits_{k=i+1}^{2i}2^{k(\alpha-\frac{qs+t+a(t-q)}{s+t-q+1})} \right)
        \end{align}
        and
        \begin{align}
            I_p \leq &C i^{-\frac{ps+t+a(t-p)}{s+t-p+1} } \left(\sum\limits_{k=i+1}^{2i}2^{k(\alpha-\frac{ps+t+a(t-p)}{s+t-p+1})} \right).
        \end{align}
        Noting that 
        \begin{align*}
            \frac{t-p}{s+t-p+1}>0,\quad \frac{t-q}{s+t-q+1}>0,
        \end{align*}
        there exists a $l_0$ depends on $\alpha,p,q,s,t$ such that if $l\geq l_0$, then $I_p\leq C$ and $I_q\leq C$. Moreover, we have
        \begin{align*}
            \int_{B_{2^{i+1}}}u^{s-l-i^{-1}}|\nabla u|^td\mu \leq C.
        \end{align*}
        Repeating the same argument as in proof of Theorem 1.2, we have
        \begin{align*}
            \int_{M}u^{s-l}|\nabla u|^td\mu=0.
        \end{align*}
        This contradicts that $u$ is nontrivial.
        
        ~\\
        (b) $(s,t)\in G_{4,b}$. Assuming that $u$ is nontrivial, then there exists a $k>0$  and a set
        \begin{align*}
            W=\{x\in M| 0< u(x)\leq k,|\nabla u|>0\}
        \end{align*}
        such that $W$ has positive measure. Hence
        \begin{align}
            \int_{W\cap B_R}|\nabla u|^td\mu\leq k^{a-s}\int_{B_R}u^{s-a}|\nabla u|^td\mu \leq Ck^a\int_Mu^{s-a}|\nabla u|^t\varphi^b_Rd\mu,
        \end{align}
        where $\varphi_R=h(\frac{r(x)}{R})$ and $a>s$. 
       Recalling (2.2), we have
       \begin{align}
              &k^a\int_M u^{s-a}|\nabla u|^t\varphi_R^bd\mu 
              \nonumber\\
              \leq& k^aC^{\frac{s+t-a}{s+t-p+1}}(2b)^{\frac{ps+t+a(t-p)}{s+t-p+1}}a^{-\frac{(p-1)s+a(t-p+1)}{s+t-p+1}}\int_M|\nabla \varphi_R|^{\frac{ps+t+a(t-p)}{s+t-p+1}}d\mu 
              \nonumber\\
              &+ k^aC^{\frac{s+t-a}{s+t-q+1}}(2b)^{\frac{qs+t+a(t-q)}{s+t-q+1}}a^{-\frac{(q-1)s+a(t-q+1)}{s+t-q+1}}\int_M|\nabla \varphi_R|^{\frac{qs+t+a(t-q)}{s+t-q+1}}d\mu .
          \end{align}
          Let $a=2R$ and $b=c_1R$ such that (2.3) holds, where $R$ large enough and $c_1>0$ does not depends on $R$. Then we have
          \begin{align*}
              &k^aC^{\frac{s+t-a}{s+t-p+1}}(2b)^{\frac{ps+t+a(t-p)}{s+t-p+1}}a^{-\frac{(p-1)s+a(t-p+1)}{s+t-p+1}}\int_M|\nabla \varphi_R|^{\frac{ps+t+a(t-p)}{s+t-p+1}}d\mu
              \\
              \leq &C^R (2R)^{-\frac{a}{s+t-p+1}}V(2R)R^{-\frac{ps+t+a(t-p)}
              {s+t-p+1}}
              \\
              \leq & C^R e^{(\kappa-\frac{t-p+1}{s+t-p+1})2R\mathrm{ln}(2R)}
              \\
              \leq & Ce^{(\kappa+\epsilon-\frac{t-p+1}{s+t-p+1})2R\mathrm{ln}(2R)},
          \end{align*}
          where $C$ does not depends on $R$ and $\epsilon>0$ to be determined later. By the same argument, it is obvious that
          \begin{align*}
              &k^aC^{\frac{s+t-a}{s+t-q+1}}(2b)^{\frac{qs+t+a(t-q)}{s+t-q+1}}a^{-\frac{(q-1)s+a(t-q+1)}{s+t-q+1}}\int_M|\nabla \varphi_R|^{\frac{qs+t+a(t-q)}{s+t-q+1}}d\mu
              \\
              \leq & Ce^{(\kappa+\epsilon-\frac{t-q+1}{s+t-q+1})2R\mathrm{ln}(2R)}.
          \end{align*}
          Combining with (3.23) and (3.24), we have
          \begin{align*}
              \int_{W\cap B_R}|\nabla u|^td\mu\leq Ce^{(\kappa+\epsilon-\frac{t-p+1}{s+t-p+1})2R\mathrm{ln}(2R)}+Ce^{(\kappa+\epsilon-\frac{t-q+1}{s+t-q+1})2R\mathrm{ln}(2R)}. 
          \end{align*}
          Since 
          \begin{align*}
            0<\kappa <\mathrm{min}\{\frac{t-p+1}{s+t-p+1},\frac{t-q+1}{s+t-q+1}\},
          \end{align*}
          then we can find a $\epsilon>0$ such that
          \begin{align*}
              \kappa+\epsilon-\frac{t-q+1}{s+t-q+1}<0,\quad \kappa+\epsilon-\frac{t-p+1}{s+t-p+1}<0.
          \end{align*}
          Letting $R\xrightarrow{}\infty$, we have
          \begin{align*}
              \int_{W}|\nabla u|^td\mu=0,
          \end{align*}
          which contradicts that the measure of $W$ is positive. This completes the proof. \qed
          
          \section{Some examples}
          In this section, we will construct the solution by patching together functions near the origin and at infinity, i.e., this proves Theorem 1.8. Actually, the example originates from \cite{SXX}, and the crux of the proof lies in demonstrating that the effect of $\Delta_pu$ is marginal.
                    
          ~\\
        $\mathit{Proof\ of\ Theorem\ 1.8\ (1)-(3)}.$
           Let $(\mathbb{R}^n,g)$ be a Riemannian manifold with metric
        \begin{align}
            g=dr^2+\varphi^2(r)d\theta^2,
        \end{align}
        where $(r,\theta)$ are the polar coordinates and $\varphi(r)$ is a smooth, positive, increasing function such that
        \begin{align*}  
    \varphi(r)=\begin{cases}
        r, & \quad \forall  r<< 1,
           \\
           (r^{\alpha-1} (\mathrm{ln}\ r)^\beta)^{\frac{1}{n-1}},  &  \quad \forall r>>1.\\
    \end{cases}
\end{align*}
           Here $\alpha\geq 1 $ and $ \beta>0$ are constants to be determined by the volume condition. Then $(\mathbb{R}^n,g)$ is a complete manifold and
       \begin{align*}  
    S(r)=\begin{cases}
        \omega_n r^{n-1}, & \quad \forall  r<< 1
           \\
           \omega_n r^{\alpha-1} (\mathrm{ln}\ r)^\beta,  &  \quad \forall  r>>1,\\
    \end{cases} 
\end{align*}
          where $S(r)$ is the surface area of the ball $B_O(r)$ centered at origin $O$ and $ \omega_n$ is the surface area of the unit ball in $\mathbb{R}^n$. In particular, the monotonicity of $S$ follows from that of $\varphi$ and the volume of the ball $B_O(r)$ satisfies the following estimate
          \begin{align*}
              V(r)=Vol(B_O(r))=\int_0^rS(y)dy\leq Cr^{\alpha} (\mathrm{ln}\ r)^\beta, \quad \forall  r>>1.
          \end{align*}
         Since we will construct a positive radial solution $v$, the inequality (1.5) is equal to
         \begin{align}
         (S|v^{'}|^{q-2}v^{'})^{'}+(S|v^{'}|^{p-2}v^{'})^{'}+Sv^s|v^{'}|^{t}\leq 0.
     \end{align}
         First, we define the positive radial function $u(r)$ by
          \begin{align*}  
    u(r)=\begin{cases}
        \frac{u_a(R_0)}{u_R(R_0)}u_R, & \quad   r\in [0,R_0],
           \\
           u_a ,  &  \quad   r\in (R_0,\infty),\\
    \end{cases} 
\end{align*}
   where 
   \begin{align*}
       u_a(r)=\int_r^{\infty} y^{-\frac{\alpha-1}{q-1}}(\mathrm{ln}\ y)^{-\frac{a}{q-1}}dy, \quad   r\in (R_0,\infty)
   \end{align*}
   and 
   \begin{align*}
         u_R(r)=\frac{R^{\gamma+1}-r^{\gamma+1}}{R^{\gamma+1}}, \quad   r\in [0,R_0].
     \end{align*}
   Here $a,\gamma>0,R_0>>1$ are constants to be determined later and $R>R_0$ is determined by the following gluing. Notice that 
   \begin{align*}
         \frac{u_a^{'}(R_0)}{u_a(R_0)}<0,\quad \lim\limits_{R\xrightarrow{}\infty}\frac{u^{'}_R(R_0)}{u_R(R_0)}=0,\quad  \lim\limits_{R\xrightarrow{}R_0^{+}}\frac{u^{'}_R(R_0)}{u_R(R_0)}=-\infty.
     \end{align*}
     Thus, we can choose $R>R_0$ such that 
     \begin{align*}
         \frac{u_a^{'}(R_0)}{u_a(R_0)}=\frac{u^{'}_R(R_0)}{u_R(R_0)}.
     \end{align*}
     This implies that $u\in C^1$. 
     \par
     For $r\in [0,R_0]$, $u=u_R$ and
     \begin{align*}
         u_R^{'}(r)=-(\gamma+1)\frac{r^\gamma}{R^{\gamma+1}}.
     \end{align*}
     Since $S^{'}\leq 0$, we have
     \begin{align*}
         (S|u_R^{'}|^{q-2}u_R^{'})^{'}=&-(S(\gamma+1)^{q-1}\frac{r^{(q-1)\gamma}}{R^{(\gamma+1)(q-1)}})^{'}
         \\
         \leq& -S(\gamma+1)^{q-1}(q-1)\gamma \frac{r^{(q-1)\gamma-1}}{R^{(\gamma+1)(q-1)}},
         \\
         (S|u_R^{'}|^{p-2}u_R^{'})^{'}\leq &-S(\gamma+1)^{p-1}(p-1)\frac{\gamma r^{(p-1)\gamma-1}}{R^{(\gamma+1)(q-1)}}\leq 0.
     \end{align*}
     Noting that $\frac{R^{\gamma+1}-R_0^{\gamma+1}}{R^{\gamma+1}}\leq u_R\leq 1$, we obtain
     \begin{align*}  
    Su_R^s|u_R^{'}|^{t}\leq\begin{cases}
        S(\gamma+1)^t\frac{r^{\gamma t}}{R^{(\gamma+1)t}}, & \quad s\geq 0,
           \\
           S(\frac{R^{\gamma+1}-R_0^{\gamma+1}}{R^{\gamma+1}})^s(\gamma+1)^t\frac{r^{\gamma t}}{R^{(\gamma+1)t}},  &  \quad s<0.
    \end{cases} 
\end{align*}
        Then the equality 
        \begin{align}
         (S|u_R^{'}|^{q-2}u_R^{'})^{'}+k(S|u_R^{'}|^{p-2}u_R^{'})^{'}+\lambda Su_R^s|u_R^{'}|^{t}\leq 0
     \end{align}
     holds for $r \in [0,R_0]$ and $0<k\leq 1$ if we select 
   \begin{align*}  
    \gamma\begin{cases}
        =c>0, & \quad q-t-1\leq  0,
           \\
           <\frac{1}{q-t-1},  &  \quad q-t-1>0
    \end{cases} 
\end{align*}
   and 
    \begin{align*}  
    0<\lambda\leq\begin{cases}
        (\gamma+1)^{q-t-1}(q-1)\gamma\frac{R_0^{\gamma(q-t-1)-1}}{R^{(\gamma+1)t}}, & \quad s\geq 0,
           \\
           (\frac{R^{\gamma+1}-R_0^{\gamma+1}}{R^{\gamma+1}})^{-s}(\gamma+1)^{q-t-1}(q-1)\gamma\frac{R_0^{\gamma(q-t-1)-1}}{R^{(\gamma+1)t}},  &  \quad s<0.
    \end{cases} 
\end{align*}
     To ensure $u^s|\nabla u|^t\in L_{loc}^{1}(B_{R_0}(O))$, we require $\gamma<{-\frac{n}{t}}$ when $t<0.$ 
     \par
     For $r>R_0>>1$, $u=u_a$ and 
     \begin{align*}
         u_a^{'}=-r^{-\frac{\alpha-1}{q-1}}(\mathrm{ln}\ r)^{-\frac{a}{q-1}}.
     \end{align*}
     To ensure that $u_a$ is finite, we need $\alpha>q$ or
     \begin{align*}
         \alpha=q, \quad a>q-1.
     \end{align*}
     From the definition of $S(r)=\omega_n r^{\alpha-1} (\mathrm{ln}\ r)^\beta$, we have
     \begin{align*}
         (S|u_R^{'}|^{q-2}u_R^{'})^{'}=&-(Sr^{-{\alpha+1}}(\mathrm{ln}\ r)^{-a})^{'}=-\omega_n(\beta-a)(\mathrm{ln}\ r)^{\beta-a-1}\frac{1}{r},
         \\
         (S|u_R^{'}|^{p-2}u_R^{'})^{'}=&-(Sr^{-\frac{(\alpha-1)(p-1)}{q-1}}(\mathrm{ln}\ r)^{-a\frac{p-1}{q-1}})^{'}
         \\
         \leq &C r^{-\frac{(\alpha-1)(p-1)}{q-1}-1}(\mathrm{ln}\ r)^{-a\frac{p-1}{q-1}},
         \\
         Su_R^s|u_R^{'}|^{t}\leq&C_0 r^{\alpha-1+s-\frac{(\alpha-1)(s+t)}{q-1}}(\mathrm{ln}\ r)^{\beta-a\frac{s+t}{q-1}}.
     \end{align*}
     Here we use $R_0>>1$ to estimate $(S|u_R^{'}|^{p-2}u_R^{'})^{'}$. Since $r>R_0>>1$, we only need to compare the leading coefficients of 
     $(S|u_R^{'}|^{q-2}u_R^{'})^{'},(S|u_R^{'}|^{p-2}u_R^{'})^{'} $ and $ Su_R^s|u_R^{'}|^{t}.$ It is easy to see that $-\frac{(\alpha-1)(p-1)}{q-1}-1<-1$.
     Hence there exists a $\lambda>0$ such that the equality 
        \begin{align}
         (S|u_R^{'}|^{q-2}u_R^{'})^{'}+k(S|u_R^{'}|^{p-2}u_R^{'})^{'}+\lambda Su_R^s|u_R^{'}|^{t}\leq 0
     \end{align}
     holds for $R_0>>1$ and $0<k\leq 1$ if   
     \begin{align*}
        \beta>a,\  \frac{\alpha-1}{q-1}(s+t-q+1)> s+1
     \end{align*}
     or
     \begin{align*}
       \beta>a,\  \frac{\alpha-1}{q-1}(s+t-q+1)= s+1,\ \frac{s+t-q+1}{q-1}a\geq 1.
     \end{align*}
     Now we begin to prove Theorem 1.8 (1)-(3). 
     
     ~\\
     (a) $(s,t)\in G_{1,a}$. Let $\alpha=\frac{qs+t}{s+t-q+1}>q $ and $\beta=\frac{q-1}{s+t-q+1}+\epsilon$ for $\epsilon>0$. Then the model manifold $(\mathbb{R}^n,g)$ satisfies the volume condition in Theorem 1.8 (1). Choosing $\frac{s+t-q+1}{q-1}<a<\beta$, then we have
     \begin{align*}
       \  \frac{\alpha-1}{q-1}(s+t-q+1)=s+1,\ \frac{s+t-q+1}{q-1}a> 1.
     \end{align*}
     This yields that there exists $R_0>>1$ such that
    \begin{align}
         (S|u_R^{'}|^{q-2}u_R^{'})^{'}+k(S|u_R^{'}|^{p-2}u_R^{'})^{'}+ Su_R^s|u_R^{'}|^{t}\leq 0,\quad r \in [R_0,\infty)
     \end{align}
     for $0<k \leq 1$. Combining with (4.3), we have 
     \begin{align}
         \Delta_qu+ k\Delta_p u+\lambda_0 u^s|\nabla u|^t\leq 0,
     \end{align}
    where $\lambda_0=min\{1,\lambda \}$. Letting $k=\lambda_0^{\frac{p-q}{s+t-q+1}}\leq 1$ and $u=\lambda_0^{-\frac{1}{s+t-q+1}}v$, we get a nontrivial positive weak solution $v$ to 
     \begin{align*}
         \Delta_qv+ \Delta_p v+v^s|\nabla v|^t\leq 0
     \end{align*}
     on $(\mathbb{R}^n,g)$.

     ~\\
   (b) $(s,t)\in G_{1,b}\bigcup (G_{2,a}\cap\{t>q\})$. Let $\alpha=q $ and $\beta=q-1+\epsilon$ for $\epsilon>0$. Choosing $q-1<a<\beta$ and using $s+t-q+1>0$, we have
       \begin{align*}
       \  \frac{\alpha-1}{q-1}(s+t-q+1)>s+1, \quad  t>q
     \end{align*}    
     or
     \begin{align*}
         \frac{\alpha-1}{q-1}(s+t-q+1)= s+1,\ \frac{s+t-q+1}{q-1}a> 1,\quad (s\geq 0,t=q)\in G_{1,b}.
     \end{align*}
     By the same argument in (1), there exists a nontrivial positive function $v$ to 
     \begin{align*}
         \Delta_qv+ \Delta_p v+v^s|\nabla v|^t\leq 0.
     \end{align*}  
     
   ~\\
   (c) $(s,t)\in (G_{3,c}\cap \{t>p-1\})\bigcup (G_{2,a}\cap\{t=q\})$. Letting $\alpha=\frac{t}{t-q+1} $ and $\beta=\frac{q-1}{t-q+1}+\epsilon$ for $\epsilon>0$, then we have 
   \begin{align*}
       V(r)\leq C r^{\frac{t}{t-q+1}}(\mathrm{ln}\ r)^{\frac{q-1}{t-q+1}+\epsilon}, \quad r>>1.
   \end{align*}
   By considering the case $s=0$ in $G_1$, it follows that the model manifold $(\mathbb{R}^n,g)$ admits a a nontrivial nonnegative solution $v$ to
   \begin{align*}
         \Delta_qv+ \Delta_p v+|\nabla v|^t\leq 0.
     \end{align*}
    Let $f=v+1$, then we obtain
    \begin{align*}
        \Delta_qf+ \Delta_p f+f^s|\nabla f|^t\leq 0
    \end{align*}
    since $s<0$.
   \qed

~\\
        $\mathit{Proof\ of\ Theorem\ 1.8\ (4)-(5)}.$ 
        \\
        (a) $(s,t)\in G_{4,a}$. For $r>>1$, letting $S(r)=e^{\lambda r}$, then $V(r)\leq Ce^{\lambda r}$. By the result in \cite{SXX}, there exists $c,d>0$ such that $u_c(r)=d+r^{-c}$ is the solution to the inequality
        \begin{align*}
            \Delta_q u+u^s|\nabla u|^t\leq 0,\quad  r>>1
        \end{align*}
      on $(\mathbb{R}^n,g)$. Note that $\Delta_p u_R\leq0$ for $r\in [0,R_0]$.  Since
      \begin{align*}
          (S|u_c^{'}|^{p-2}u_c^{'})^{'}=&-c^{p-1}(e^{\lambda r}r^{-(c+1)(p-1)})^{'}
          \\
          =&-c^{p-1}e^{\lambda r}r^{-(c+1)(p-1)}(\lambda-\frac{(c+1)(p-1)}{r}),
      \end{align*}
      there exists $R_0$ such that $(S|u_c^{'}|^{p-2}u_c^{'})^{'}\leq 0 $ for $r>R_0$, i.e., $\Delta_p u_c\leq0 $. Thus we can glue the two parts to obtain a solution to (1.5).
      
      ~\\
      (b) $(s,t)\in G_{4,b}$. For $r>>1$, letting $S(r)=e^{\lambda r^\gamma\mathrm{ln}r}$, then $V(r)<Ce^{\lambda r^\gamma\mathrm{ln}r}$. From \cite{SXX}, the function $u(r)=\frac{\mathrm{ln}r}{r}$ satisfies
        \begin{align*}
            \Delta_q u+u^s|\nabla u|^t\leq 0,\quad  r>>1
        \end{align*}
      on $(\mathbb{R}^n,g)$. As in the above, we can prove that 
      \begin{align*}
          \Delta_p u(r)\leq0, \quad  r>>1.
      \end{align*}
      Hence we obtain a positive solution to 
      \begin{align*}
            \Delta_q u+\Delta_pu+u^s|\nabla u|^t\leq 0.
        \end{align*}
        \qed
\bibliographystyle{siam}
\bibliography{ref}

\end{document}